\documentclass{amsart}
\usepackage{graphicx} 

\usepackage{amsfonts}%
\usepackage{amsmath}%
\usepackage{amssymb}%
\usepackage{graphicx}
\usepackage{cite}
\usepackage{geometry}
\usepackage{enumerate}
\usepackage[table]{xcolor}
\usepackage{multirow}
\usepackage{enumitem}
\usepackage{nicematrix}
\usepackage{hhline}
\usepackage{booktabs}
\usepackage{tabularx}
\usepackage{mathrsfs}
\usepackage{wrapfig}
\usepackage{caption}
\usepackage{subcaption}
\usepackage[justification=centering]{caption}

\usepackage{chemformula}

\providecommand{\U}[1]{\protect \rule{.1in}{.1in}}
\newtheorem{theorem}{Theorem}
\theoremstyle{plain}

\newtheorem{corollary}{Corollary}

\newtheorem{definition}{Definition}
\newtheorem{example}{Example}

\newtheorem{lemma}{Lemma}

\newtheorem{remark}{Remark}

\numberwithin{equation}{section}

\title[Korovkin-type approximation theorem]{Power Series Statistical Convergence and An Abstract Korovkin-type approximation theorem}
\author{D\.{I}lek S\"{o}ylemez}
\author{Mehmet \"{U}nver$^*$}

\keywords{Korovkin-type approximation theorem, power series statistical convergence, $r$-th
order generalization   
\\
$^{a}$Department of Mathematics, Faculty of Science, Selçuk University,
Selçuklu, 42003 Konya, Türkiye\\
$^{b}$Department of Mathematics, Faculty of Science, Ankara University,
Ankara, Türkiye\\ dilek.soylemez@selcuk.edu.tr, munver@ankara.edu.tr  \\ *Corresponding author}

\begin{document}
\begin{abstract}
This paper establishes an abstract Korovkin-type approximation theorem in general spaces, extending the framework of approximation theory to accommodate broader contexts. A critical result supporting this theorem is the proof that any $P$-statistically convergent sequence contains a classically convergent subsequence over a density $1$ set, which plays a foundational role in the analysis. As a conclusion, we investigate the convergence of the $r$-th order generalization of linear operators, which may lack positivity, and present a Korovkin-type approximation theorem for periodic functions, both utilizing $P$-statistical convergence. These contributions generalize and improve existing results in approximation theory, providing novel insights and methodologies, supported by practical examples and corollaries.

\end{abstract}
\maketitle

\section{Introduction}

Korovkin-type approximation theory provides essential criteria for the convergence of sequences of positive linear operators in specific ways \cite{korovkin}. This theory is driven by several key motivations. One primary goal is to identify appropriate conditions under which any sequence of positive linear operators, transitioning from one specific space to another, converges. Another aim is to explore specific convergence conditions for certain sequences of positive linear operators using established criteria (see, e.g., \cite{altomare}). Additionally, the theory introduces concepts from summability, which aims to induce convergence in general senses for sequences or series that do not converge in the conventional sense. In this significant contribution to the field, Gadjiev and Orhan \cite{gadjiev-orhan} set forth criteria for the statistical convergence of sequences of positive linear operators within \(C[a,b]\), the space of all real-valued continuous functions defined on the interval \([a,b]\). Recent advancements in summability methods and their applications in approximation theory explored in a variety of studies. Alemdar and Duman \cite{alemdar} studied general summability methods with Bernstein–Chlodovsky operators, improving the understanding of their approximation capabilities. Atlıhan and Orhan \cite{atlihan} investigated matrix summability in conjunction with positive linear operators, expanding the theoretical framework for these methods. Duman \cite{duman2007} presented Korovkin-type approximation theorems using ideal convergence, offering a novel approach to this classical theory. Taş and Yurdakadim \cite{tas-yurdakadim2016} explored approximation methods for derivatives of functions in weighted spaces using linear operators and power series methods, contributing to the refinement of approximation techniques. Sakaoğlu and Ünver \cite{sakaoglu-unver} examined statistical approximation for multi-variable integrable functions, providing empirical insights into statistical methods in approximation. Söylemez and Ünver \cite{soylemez-un2021} discussed the rates of power series statistical convergence of positive linear operators, further bridging the gap between summability and operator theory. In a similar vein, Söylemez and Ünver \cite{soy-unver2017} applied summability methods to Cheney–Sharma operators, presenting Korovkin-type approximation theorems that improve the applicability of these operators. Srivastava et al. \cite{srivastava2021} linked approximation theory and summability methods through the study of four-dimensional infinite matrices, offering a comprehensive view of the interplay between these mathematical concepts.

The power series method is a function-theoretic approach that is particularly suitable for applications involving analytic continuation and the numerical solutions of systems of linear equations (see \cite{boos}, Sections 5.2 and 5.3). Such methods of a function-theoretical nature were proven highly effective in Korovkin-type approximation theory. The first Korovkin-type approximation theorem via Abel convergence, a specific power series method, was introduced by Ünver in \cite{unver}. The concept of statistical convergence, first introduced by Fast \cite{fast}, is a significant topic within summability theory and has found numerous applications in Korovkin-type approximation theory. Ünver and Orhan \cite{unver-orhan2019} further extended this concept by developing a version of statistical convergence using power series summability methods, called power series statistical convergence. Building on this foundational work, numerous researchers extended the theory, deriving Korovkin-type approximation results using various power series methods (see, e.g., \cite{braha2, unver3, demirci4, Soylemez2021, ulucay-un}). A concise progression of the theory up to Korovkin-type approximation theorems using power series statistical convergence is illustrated in Figure \ref{dev}.

\begin{figure}[ht]
    \centering
\includegraphics[scale=0.2]{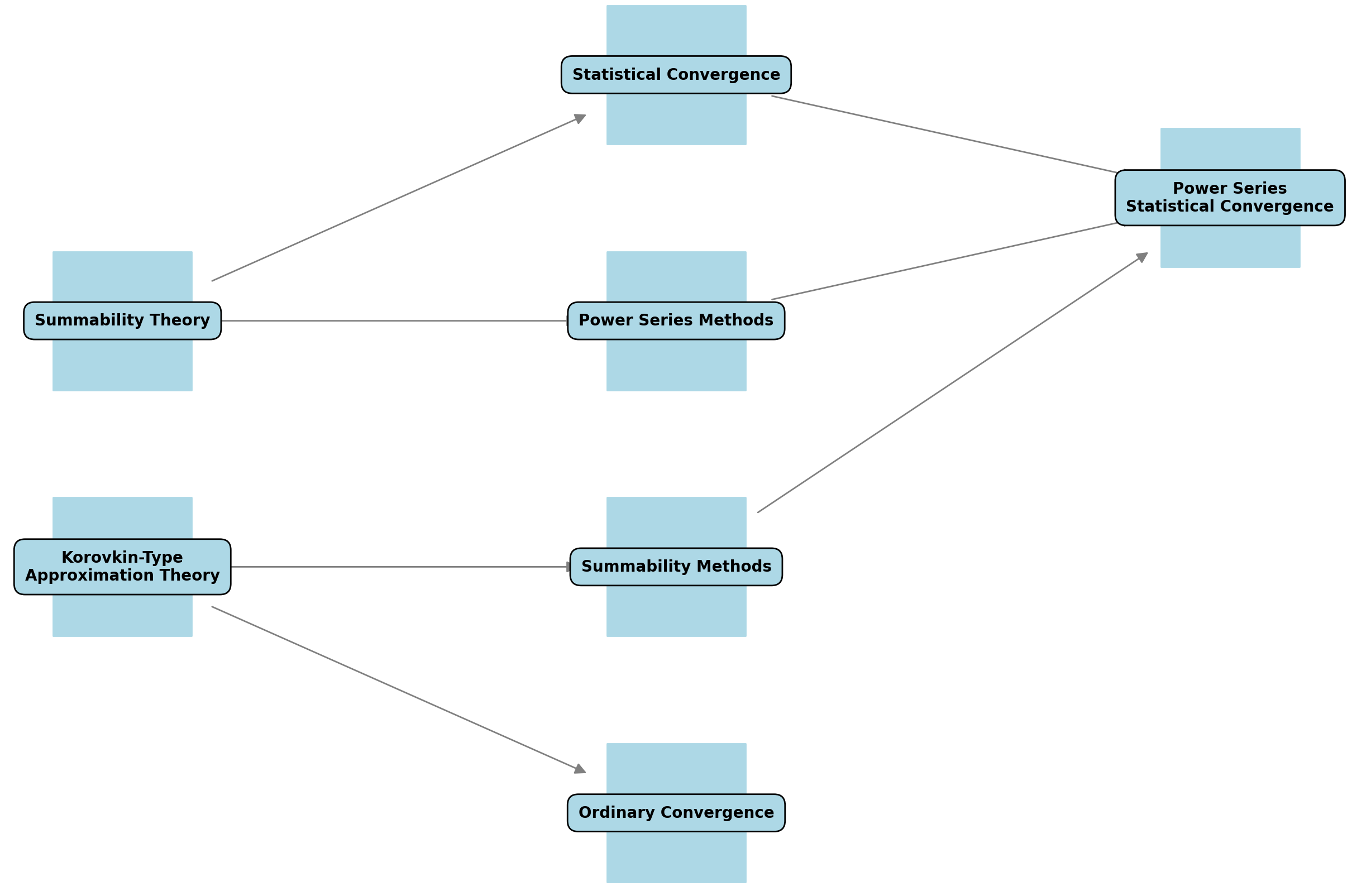}
    \caption{Development of Korovkin-type approximation theorems up to power series statistical convergence}
    \label{dev}
\end{figure}

The Korovkin theorem has been extensively generalized to various contexts and function spaces, including abstract spaces (see \cite{mhaskar-pai}). The abstract versions of the Korovkin theorem encompass and extend many of the classical results found in the literature, as well as their multivariate counterparts, making these results more comprehensive and widely applicable. Such generalizations provide a unified framework that is both more robust and versatile, allowing for broader applications in functional analysis and approximation theory. Additionally, abstract Korovkin theorems have been developed using a variety of summability methods, including statistical convergence, power series convergence, and the $A$-summation process. Each of these methods introduces unique perspectives and tools for analyzing operator behavior in abstract spaces. Notable contributions in this direction include the works of \cite{dumanabstract}, \cite{tas-abstract}, and \cite{atlihan-tas}, which demonstrate the utility of these summability techniques in extending the classical Korovkin framework to more complex and generalized settings. These advancements highlight the evolving nature of approximation theory, driven by the interplay between operator theory and summability methods.


It is well known that any statistically convergent sequence has a convergent subsequence over a density $1$ set, a property that has been extensively utilized in the study of statistical and $A$-statistical convergence. This result provides a fundamental tool for analyzing the behavior of sequences in these frameworks, allowing for the reduction of complex convergence phenomena to simpler, more manageable components. Traditionally, such results focus on identifying subsets of indices where the sequence exhibits classical convergence, revealing deeper insights into the structure and nature of statistical convergence. The work of Šalát in 1980 \cite{salat} was pivotal in formalizing these ideas, particularly in the context of classical statistical convergence (see also \cite{103, 102}). Šalát demonstrated how these principles could be applied to refine the understanding of statistical convergence, laying the groundwork for subsequent generalizations.

Building on this foundation, we adapt these classical principles to the $P$-statistical convergence framework, which offers a more inclusive perspective on convergence by incorporating power series methods. Our theorem extends these traditional results to the $P$-statistical setting, providing a novel approach to understanding the convergence of sequences through the lens of density and power series methods.
We also present an abstract Korovkin-type approximation theorem, a significant contribution that generalizes classical results to broader functional spaces. As a conclusion, we establish the convergence of $r$-th order generalizations of positive linear operators, even in cases lacking positivity. Additionally, we develop a Korovkin-type approximation theorem for periodic functions utilizing $P$-statistical convergence. These findings not only reinforce but also advance prior research, paving the way for new applications and deeper insights into approximation theory.

The rest of the paper is organized as follows. In Section \ref{sec2}, we introduce the fundamental concepts and definitions related to power series statistical convergence, including regularity conditions and $P$-density. Section \ref{sec3} presents the main results, focusing on a key property of $P$-statistical convergence and its role in proving an abstract Korovkin-type approximation theorem. Section \ref{sec4} explores consequences of these results, including the convergence of $r$-th order generalizations of positive linear operators and a Korovkin-type approximation theorem for periodic functions. Finally, in Section \ref{sec5}, we provide concluding remarks and outline potential directions for future research, emphasizing the broader implications of our findings in approximation theory and summability methods.

\section{Preliminaries}\label{sec2}
We begin this section by recalling the definition of power series convergence, which is foundational to the concept of power series statistical convergence. This involves the use of a regular power series method, defined by its sequence coefficients and radius of convergence. Next, we recall the regularity conditions necessary and sufficient for such methods and the notion of $P$-density, which plays a key role in defining $P$-statistical convergence. Additionally, we revisit important auxiliary concepts, such as the modulus of continuity, Lipschitz classes, and the framework of positive linear operators in the context of Banach spaces, setting the stage for the main results in the paper.


\begin{definition}\cite{boos}
Let $P=\left(  s_{j}\right)  $ be a real sequence with $s_{0}>0$ and
$s_{1},s_{2},...\geq0,$ and such that the corresponding power series $s(t)=%
{\displaystyle \sum \limits_{j=0}^{\infty}}
s_{j}t^{j}$ has radius of convergence $R$ with $0<R\leq \infty.$ If
\[
\lim_{0<t\rightarrow R^{-}}\frac{1}{s(t)}\sum_{j=0}^{\infty}x_{j}s_{j}%
t^{j}=L\text{,}%
\]
then we say that $x=\left(  x_{j}\right)  $ is convergent in the sense of
power series method $P$ ($P$ convergent).
\end{definition}

\color{black}

\begin{theorem}
\cite{boos} A power series method $P$ is regular if and only if for any $j\in%
\mathbb{N}
_{0}$%
\[
\lim_{0<t\rightarrow R^{-}}\frac{s_{j}t^{j}}{s(t)}=0,
\]
where $\mathbb{N}_0$ is the set of all non-negative integers.
\end{theorem}
Next, we recall the concept of power series statistical convergence, as defined in \cite{unver-orhan2019}. This concept is formulated using the notion of density with respect to power series methods and is distinct from both power series convergence and statistical convergence, making direct comparisons infeasible.


\begin{definition}
\cite{unver-orhan2019} Let $P$ be a regular power series method and let
$E\subset%
\mathbb{N}
_{0}.$ If
\[
\lim_{0<t\rightarrow R^{-}}\frac{1}{s(t)}%
\sum_{j\in E}^{\infty}s_{j}t^{j}%
\]
exists, then this limit is called the $P$-density of $E$ and denoted by $\delta_{P}\left(  E\right)  $.
\newline
\end{definition}

\begin{definition}
\cite{unver-orhan2019} Let $x=(x_{j})$ be a real sequence and let $P$ be a regular
power series method. Then $x$ is said to be $P$-statistically convergent to
$L$ if for any $\varepsilon>0$%
\[
\lim_{0<t\rightarrow R^{-}}\frac{1}{s(t)}\sum_{j:\left \vert x_{j}-L\right \vert
\geq \varepsilon}s_{j}t^{j}=0,
\]
i.e., $\delta_{P}\left(  \{j\in \mathbb{N}_{0}:\left \vert x_{j}-L\right \vert
\geq \varepsilon \} \right)  =0$. In this case we write $st_{P}-\lim\limits_{j\to\infty} x=L.$
\end{definition}

We now recall the modulus of continuity and the Lipschitz class, which play a critical role in analyzing the smoothness and continuity of functions( see; \cite{devore-lorentz}). The modulus of continuity, $\omega \left( \psi, \delta \right)$, for a function $\psi$ is defined as:
\[
\omega \left( \psi, \delta \right) = \sup_{\substack{|x-y| \leq \delta \\ x,y \in [0,A]}} \left| \psi(x) - \psi(y) \right|.
\]
It is well-known that for any $\psi \in C\left[ 0, A \right]$
\begin{equation*}
\lim_{\delta \to 0^+} \omega \left( \psi, \delta \right) = 0,
\end{equation*}
and for any $\delta > 0$
\begin{equation*}
\left| \psi(x) - \psi(y) \right| \leq \omega \left( \psi, \delta \right) \left( \dfrac{|x-y|}{\delta} + 1 \right).
\end{equation*}
A function $\psi\in C\left[ 0, A \right]$ is said to be Lipschitz continuous of order $\alpha$ if
\[
\left| \psi(x) - \psi(y) \right| \leq M |x-y|^\alpha,
\]
for all $x, y \in \left[ 0, A \right]$, where $M > 0$ and $0 < \alpha \leq 1$. The set of all such functions is denoted by $Lip_M\left( \alpha \right)$ \cite{devore-lorentz}. Throughout the paper, we assume that $X$ is a compact Hausdorff space with at least two points, and $C\left( X \right)$ denotes the space of all real-valued continuous functions on $X$. This is a Banach space with the norm defined by
\[
\|\psi\| = \sup_{x \in X} |\psi(x)|.
\]
We suppose $\psi_1, \psi_2, \dots, \psi_d \in C\left( X \right)$ have the following properties: There exist $g_1, g_2, \dots, g_d \in C\left( X \right)$ such that for every $x, y \in X$
\begin{equation}
\breve{Q}_x(y) = \sum_{\lambda=1}^d g_\lambda(x)\psi_\lambda(y) \geq 0, \label{5a}
\end{equation}
and
\begin{equation}
\breve{Q}_x(y) = 0 \iff y = x. \label{5b}
\end{equation}
We also assume that $(T_j)$ is a sequence of positive linear operators from $C\left( X \right)$ into itself satisfying
\begin{equation}
st_{P}-\lim_{j \to \infty} \left\| T_j\left( \psi_\lambda \right) - \psi_\lambda \right\| = 0, \label{5c}
\end{equation}
for $\lambda = 1, 2, \dots, d$. Additionally, we suppose $T_0(\psi) = 0$ for any $\psi \in C\left( X \right)$.

\section{Main Results}\label{sec3}
This section is dedicated to presenting the main theoretical contributions of the paper. We begin by establishing a key result for power series statistical convergence, which serves as a foundational tool for understanding the nuanced behavior of sequences under this convergence method. This result provides a framework to analyze sequences by identifying subsets with specific properties, facilitating detailed analysis and interpretation. Building on this foundation, we then prove an abstract Korovkin-type approximation theorem in the context of power series statistical convergence. This theorem extends classical Korovkin results to a more generalized setting, applicable to abstract spaces and operators that satisfy the conditions of $P$-statistical convergence. The results encompass a wide range of cases, including those where positivity or other traditional operator properties may not hold, demonstrating the robustness and versatility of the power series statistical framework.

\subsection{A result on $P$-statistical convergence}

In the following theorem, drawing on the approaches of \cite{top} and \cite{salat}, we present an important result related to $P$-statistical convergence analogous to classical theory. Remark 3.6.3 of \cite{boos} indicates that it suffices to consider only the cases $R=1$ and $R=\infty$ for power series methods, known as Abel type and Borel type methods, respectively.

\begin{theorem}\label{dec}
A sequence $x=(x_j)$ is $P$-statistically convergent to $L$ if and only if there exists a subset $E \subset \mathbb{N}_0$ such that $\delta_P(E) = 0$ and $\displaystyle\lim_{j \in E^c} x_j = L$.
\end{theorem}

\begin{proof}
Assume $P$ is a Borel type power series method. Define for any positive integer $k$ the set
\[
U_k = \left\{ y \in \mathbb{R} : |y - L| < \dfrac{1}{k} \right\},
\]
noting that $U_k \supset U_{k+1}$ for each $k$, where $\mathbb{R}$ is the set of all real numbers. Select positive integers $n_k$ such that
\[
\frac{1}{s(t)} \sum_{x_j \notin U_k} s_j t^j < \dfrac{1}{k^2}
\]
for all $t \geq n_k$, and without loss of generality assume $n_k < n_{k+1}$. Define $E_k$ as the set of integers $n_k \leq j < n_{k+1}$ where $x_j \notin U_k$, and set $E = \bigcup_{k=1}^{\infty} E_k$. For any $\varepsilon > 0$, there exists an integer $m$ such that
\[
\sum_{k \geq m} \dfrac{1}{k^2} < \varepsilon.
\]
Thus, for any $t \geq n_m$, we have
\begin{align*}
\frac{1}{s(t)} \sum_{j \in E} s_j t^j &\leq \dfrac{1}{s(t)} \sum_{j < n_m} s_j t^j + \sum_{k = m}^{\infty} \left( \dfrac{1}{s(t)} \sum_{j \in E, n_k \leq j < n_{k+1}} s_j t^j \right) \\
&< \dfrac{1}{s(t)} \sum_{j < n_m} s_j t^j + \varepsilon.
\end{align*}
Since $P$ is regular and the sum is finite, we have $\displaystyle\lim_{t \to \infty} \dfrac{1}{s(t)} \sum_{j < n_m} s_j t^j = 0$. Therefore, $\delta_P(E) = 0$. Also, for any $k \geq K$, where $K$ is a sufficiently large integer, we have
\[
U_k \subset \left\{ y \in \mathbb{R} : |y - L| < \varepsilon \right\}.
\]
Thus, whenever $|x_j - L| \geq \varepsilon$, it follows that $x_j$ is not in $U_k$ for any $k \geq K$. Therefore, $j$ belongs to $E$ only if $j \geq n_K$, which implies that if $j$ is in $E^c$, then $j < n_K$. Consequently, if $|x_j - L| \geq \varepsilon$, there can only be a finite number of such $j$ values in $E^c$. So, we get $\displaystyle\lim_{j \in E^c} x_j = L$. For an Abel type method, assume the condition
\[
\lim_{t \to 1^-} \dfrac{1}{s(t)} \sum_{j \notin U_k} s_j t^j = 0.
\]
Define $n_k$ such that for $u = \dfrac{t}{1 - t}$, we have
\[
\lim_{u \to \infty} \dfrac{1}{s\left(\frac{u}{1+u}\right)} \sum_{j \notin U_k} s_j \left(\frac{u}{1+u}\right)^j = 0.
\]
This ensures that as $t$ approaches 1 from the left (equivalently as $u \to \infty$), the contributions from outside $U_k$ vanish. We follow a similar procedure to define $E_k$ and demonstrate that $\delta_P(E) = 0$ and $\displaystyle\lim_{j \in E^c} x_j = L$, mirroring the Borel type method's argument under the Abel type's specific conditions.
\\
Conversely, suppose there exists $E \subset \mathbb{N}_0$ with $\delta_P(E) = 0$ and $\displaystyle\lim_{j \in E^c} x_j = L$. Let $E^c = \{i \in \mathbb{N} : j_i\}$ with $\delta_P(E^c) = 1$. Then, there exists an integer $i_0$ such that $|x_{j_i} - L| < \varepsilon$ for all $i \geq i_0$. Setting
\[
K_{\varepsilon} = \{ j \in \mathbb{N}_0 : |x_j - L| \geq \varepsilon \},
\]
we find that
\[
K_{\varepsilon} \subset \mathbb{N}_0 \setminus \{j_{i_0+1}, j_{i_0+2}, \dots\},
\]
which has $P$-density $0$. Hence, $\delta_P(K_{\varepsilon}) = 0$, completing the proof.
\end{proof}
\begin{remark}
    Similar results have been discussed in \cite{mustafa} and \cite{nilay}. In \cite{mustafa}, a related result was established in the context of ideals, while \cite{nilay} presented a decomposition theorem for $P$-statistical convergence. However, Theorem \ref{dec} provides the first direct proof of this result specifically for $P$-statistical convergence.
\end{remark}

\subsection{An abstract Korovkin-type approximation theorem via $P$-statistical convergence}

The following lemmas are required to establish the abstract Korovkin theorem.

\begin{lemma}
\label{lemma1}Suppose that $P$ is a regular power series method, (\ref{5a})-(\ref{5c}) are satisfied, 
$\varrho_{1},\varrho_{2},...,\varrho_{d}\in%
\mathbb{R}
$ and $y\in X$. Then for the function $\breve{Q}$ defined by
\begin{equation}
\breve{Q}(y)=\sum \limits_{\lambda=1}^{d}\varrho_{\lambda}\psi_{\lambda}(y)
\label{5c1}%
\end{equation}
we have
\[
st_{P}-\lim_{j\rightarrow \infty}\left \Vert T_{j}\left(  \breve{Q}\right)
-\breve{Q}\right \Vert =0.
\]

\end{lemma}

\begin{proof}
Since the sequence ($T_{j})$ is the sequence of positive linear operators, we have
\begin{align*}
\left \vert T_{j}\left(  \breve{Q};x\right)  -\breve{Q}(x)\right \vert  &
=\left \vert \sum \limits_{\lambda=1}^{d}\varrho_{\lambda}T_{j}(\psi_{\lambda
};x)-\sum \limits_{\lambda=1}^{d}\varrho_{\lambda}\psi_{\lambda}(x)\right \vert
\nonumber \\
&  \leq \sum \limits_{\lambda=1}^{d}\left \vert \varrho_{\lambda}\right \vert
\left \vert T_{j}(\psi_{\lambda};x)-\psi_{\lambda}(x)\right \vert.
\end{align*}
Let $M=\max\limits_{1\leq \lambda \leq d}\left \vert \varrho_{\lambda}\right \vert .$ If
$M=0$, then the proof is fulfilled easily. If $M>0$, then we get
\[
\left \Vert T_{j}\left(  \breve{Q}\right)  -\breve{Q}\right \Vert \leq
M\sum \limits_{\lambda=1}^{d}\left \Vert T_{j}\left(  \psi_{\lambda}\right)
-\psi_{\lambda}\right \Vert .
\]
For a given $\varepsilon>0$ let's define
\[
N=\left \{  j:\sum \limits_{\lambda=1}^{d}\left \Vert T_{j}\left(  \breve
{Q}\right)  -\breve{Q}\right \Vert \geq\varepsilon \right \}
\]and
\[
N_{\lambda}=\left \{  j:\sum \limits_{\lambda=1}^{d}\left \Vert T_{j}\left(
\psi_{\lambda}\right)  -\psi_{\lambda}\right \Vert \geq \dfrac{\varepsilon}%
{dM}\right \}
\]
for $\lambda=1,2...,d.$ Then we see that $N\subset \displaystyle\bigcup\limits_{\lambda=1}%
^{d}N_{\lambda}.$ Therefore, we get

\begin{align*}
0  &  \leq \delta_{p}\left(  \left \{  j\in%
\mathbb{N}
:\left \Vert T_{j}\left(  \breve{Q}\right)  -\breve{Q}\right \Vert
\geq \varepsilon \right \}  \right) \\
&  \leq \sum \limits_{\lambda=1}^{d}\delta_{P}\left(  \left \{  j\in%
\mathbb{N}
:\left \Vert T_{j}\left(  \psi_{\lambda}\right)  -\psi_{\lambda}\right \Vert
\geq \dfrac{\varepsilon}{dM}\right \}  \right)  .
\end{align*}
Hence, from (\ref{5c}), we get%
\[
st_{P}-\lim_{j\rightarrow \infty}\left \Vert T_{j}\left(  \breve{Q}\right)
-\breve{Q}\right \Vert =0,
\]
which finishes the proof.
\end{proof}

\begin{lemma} \label{lemma2}
Suppose that $P$ is a regular power series method and (\ref{5a})-(\ref{5c}) are satisfied. Then we get%
\[
st_{p}-\lim_{j\rightarrow \infty}\left(  \max_{x\in X}\left \vert T_{j}\left(
\breve{Q}_{x};x\right)  \right \vert \right)  =0.
\]

\end{lemma}

\begin{proof}
From (\ref{5b}), $\breve{Q}_{x}(x)=0$, so we can write%
\begin{align*}
\left \vert T_{j}\left(  \breve{Q}_{x};x\right)  \right \vert  &  =\left \vert
T_{j}\left(  \sum \limits_{\lambda=1}^{d}g_{\lambda}\psi_{\lambda};x)\right)
-\sum \limits_{\lambda=1}^{d}g_{\lambda}(x)\psi_{\lambda}(x)\right \vert \\
&  \leq \sum \limits_{\lambda=1}^{d}\left \vert g_{\lambda}(x)\right \vert
\left \vert T_{j}(\psi_{\lambda};x)-\psi_{\lambda}(x)\right \vert.
\end{align*}
Taking into account the continuity of the function $g_{\lambda}$ on $X$, we define $M' = \max\limits_{1 \leq \lambda \leq d} |g_{\lambda}|< \infty$. Consequently, we obtain
\[
\max_{x\in X}\left \vert T_{j}\left(  \breve{Q}_{x};x\right)  \right \vert \leq
M^{^{\prime}}\sum \limits_{\lambda=1}^{d}\left \Vert T_{j}\left(  \psi_{\lambda
}\right)  -\psi_{\lambda}\right \Vert.
\]
From (\ref{5c}), we get%
\[
st_{p}-\lim_{j\rightarrow \infty}\left(  \max_{x\in X}\left \vert T_{j}\left(
\breve{Q}_{x};x\right)  \right \vert \right)  =0.
\]

\end{proof}

\begin{lemma}
\label{lemma3}Suppose that $P$ is a regular power series method and  (\ref{5a})-(\ref{5c}) are satisfied. Then there exists $E\subset
\mathbb{N}_0
$ such that $\delta_{P}(E)=1$ and
\[
\sup_{j\in E}\left \Vert T_{j}\right \Vert <\infty.
\]
Here $\left \Vert T_{j}\right \Vert =\sup\limits_{\left \Vert \psi \right \Vert
=1}\left \Vert T_{j}(\psi)\right \Vert $.
\end{lemma}

\begin{proof}
Let fix two different points $s,t\in X.$ Taking into account (\ref{5a}), we define a
function $Z$ by%
\begin{equation}
Z(y)=\breve{Q}_{s}(y)+\breve{Q}_{t}(y), \label{5e}%
\end{equation}
for $y\in X.$ It is clear that for all $y\in X,$ $Z(y)>0$ which implies
that$\dfrac{1}{Z}\in C\left(  X\right)  .$ Also taking $\varrho_{\lambda}=g_{\lambda}(s)+g_{\lambda}(t),$
($\lambda=1,2,...,d)$, where each $g_{\lambda}$ is the function used in (\ref{5a}), the function $Z$
has form (\ref{5c1}). Since $\dfrac{1}{Z(y)}\leq \left \Vert \dfrac{1}%
{Z}\right \Vert $ for all $y\in X,$ we have%
\[
1\leq \left \Vert \dfrac{1}{Z}\right \Vert Z(y),
\]
which implies
\begin{equation}
\left \vert T_{j}(1;x)\right \vert \leq \left \Vert \dfrac{1}{Z}\right \Vert
\left \vert T_{j}(Z;x)\right \vert \label{5e1}%
\end{equation}
for each $x\in X$ and $j\in%
\mathbb{N}_0.  
$ By using (\ref{5e1}) we obtain%

\begin{equation}
\left \Vert T_{j}\right \Vert =\left \Vert T_{j}(1)\right \Vert \leq \left \Vert
\frac{1}{Z}\right \Vert \left \Vert T_{j}\left(  Z\right)  \right \Vert .
\label{5f}%
\end{equation}
From Lemma \ref{lemma1}, we have $st_{P}-\lim\limits_{j\rightarrow \infty}\left \Vert
T_{j}\left(  Z\right)  \right \Vert =\left \Vert Z\right \Vert $. So we can write%
\begin{equation}
\sup_{j\in E}\left \Vert T_{j}\left(  Z\right)  \right \Vert <\infty \label{5g}.%
\end{equation}
From (\ref{5f}) and (\ref{5g}), the proof is completed.
\end{proof}

\begin{lemma}
\label{lemma4}Suppose that $P$ is a regular power series method and (\ref{5a})-(\ref{5c}) are satisfied. For a fixed $x\in X$ let $\sigma_{x}:X\rightarrow%
\mathbb{R}
$ be continuous function such that $\sigma_{x}(x)=0$. Then we have
\[
st_{P}-\lim_{j\rightarrow \infty}\left(  \max_{x\in X}\left \vert T_{j}\left(
\sigma_{x};x\right)  \right \vert \right)  =0.
\]

\end{lemma}

\begin{proof}
Since $\sigma_{x}$ is continuous at $x$, for
any $\varepsilon>0$, there exists an open neighborhood $\acute{U}$ of $x$
such that $\left \vert \sigma_{x}(y)\right \vert <\varepsilon$ whenever
$y\in \acute{U}.$ Let $m=\min\limits_{y\in X\backslash \acute{U}}\breve{Q}_{x}(y)$,
$\breve{Q}_{x}$ is given by (\ref{5a}). Due to the compactness of the set
$X\backslash \acute{U}$, $m>0$. Let $K=\max\limits_{x,y\in X}\left \vert
\sigma_{x}(y)\right \vert $. Thus for all $y\in \acute{U}$, we get%
\begin{equation}
\left \vert \sigma_{x}(y)\right \vert <\varepsilon \label{5j}%
\end{equation}
and for all $y\in X\backslash \acute{U},$ we have
\begin{equation}
\left \vert \sigma_{x}(y)\right \vert \leq K\leq \dfrac{K}{m}\breve{Q}%
_{x}(y).\label{5k}%
\end{equation}
From (\ref{5j}) and (\ref{5k}), we have%
\begin{equation}
\left \vert \sigma_{x}(y)\right \vert \leq \varepsilon+\frac{K}{m}\breve
{Q}_{x}(y).\label{5l}%
\end{equation}
By applying the operator ($T_{j})$ to (\ref{5l}), we have
\[
\left \vert T_{j}\left(  \sigma_{x};x\right)  \right \vert \leq \varepsilon
T_{j}\left(  1;x\right)  +\frac{K}{m}\left \vert T_{j}\left(  \breve{Q}%
_{x};x\right)  \right \vert ,
\]
which implies%
\[
\max_{x\in X}\left \vert T_{j}\left(  \sigma_{x};x\right)  \right \vert
\leq \varepsilon \left \Vert T_{j}\right \Vert +\frac{K}{m}\max_{x\in X}\left \vert
T_{j}\left(  \breve{Q}_{x};x\right)  \right \vert .
\]
By Lemma \ref{lemma3} there exists a subset $E$ of $\mathbb{N}_0$ such that
$\delta_{P}(E)=1$ and for all $j\in E$ we have%
\[
\max_{x\in X}\left \vert T_{j}\left(  \sigma_{x};x\right)  \right \vert \leq
A\varepsilon+\frac{K}{m}\max_{x\in X}\left \vert T_{j}\left(  \breve{Q}%
_{x};x\right)  \right \vert ,
\]
where $A=\sup\limits_{j\in E}\left \Vert T_{j}\right \Vert <\infty.$ For a given $r>0$ without loss of generality we can assume $A\varepsilon<r.$ Let us define 
\[
D_{1}=\left \{  j\in E:\max_{x\in X}\left \vert T_{j}\left(  \sigma
_{x};x\right)  \right \vert \geq r\right \}
\]
and
\[
D_{2}=\left \{  j\in E:\max_{x\in X}\left \vert T_{j}\left(  \breve{Q}%
_{x};x\right)  \right \vert \geq \dfrac{\left(  r-A\varepsilon \right)  m}%
{K}\right \}  .
\]
Then we see that $D_{1}\subset D_{2}.$ Therefore, we get
\begin{align*}
0 &  \leq \delta_{P}\left(  \left \{  j\in E:\max_{x\in X}\left \vert
T_{j}\left(  \sigma_{x};x\right)  \right \vert \geq r\right \}  \right)  \\
&  \leq \delta_{P}\left(  \left \{  j\in E:\max_{x\in X}\left \vert T_{j}\left(
\breve{Q}_{x};x\right)  \right \vert \geq \dfrac{\left(  r-A\varepsilon \right)
m}{K}\right \}  \right).
\end{align*}
From Lemma \ref{lemma2}, we have
\[
st_{P}-\lim_{j\rightarrow \infty}\left(  \max_{x\in X}\left \vert T_{j}\left(
\sigma_{x};x\right)  \right \vert \right)  =0.
\]
This finishes the proof.
\end{proof}

We are ready to prove an abstract Korovkin-type
approximation theorem.

\begin{theorem}
\label{theorem2}Assume that $P$ is a regular power series method and (\ref{5a})-(\ref{5c}) are held,
then for all $\psi \in C\left(  X\right)$, we have%
\[
st_{P}-\lim_{j\rightarrow \infty}\left \Vert T_{j}\left(  \psi \right)
-\psi \right \Vert =0.
\]

\end{theorem}

\begin{proof}
For a fixed $x\in X$, define the function $\sigma_{x}$ on $X$ by
\[
\sigma_{x}(y)=\psi(y)-\frac{\psi(x)}{Z(x)}Z(y), \]
where $Z$ is the function given by (\ref{5e}).
Then we have
\[
\left \vert T_{j}\left(  \psi;x\right)  -\psi(x)\right \vert \leq \left \vert
T_{j}\left(  \sigma_{x};x\right)  \right \vert +\frac{\left \vert \psi
(x)\right \vert }{Z(x)}\left \vert T_{j}\left(  Z;x\right)  -Z(x)\right \vert ,
\]
which yields that%
\[
\left \Vert T_{j}\left(  \psi \right)  -\psi \right \Vert \leq C\left \{
\max_{x\in X}\left \vert T_{j}\left(  \sigma_{x};x\right)  \right \vert
+\left \Vert T_{j}\left(  Z\right)
-Z\right \Vert \right \}  ,
\]
where $C=\max \left \{  1,\left \Vert \dfrac{\psi}{Z}\right \Vert \right \}$. Now
let us define
\[
G=\left \{  j\in%
\mathbb{N}
:\left \Vert T_{j}\left(  \psi \right)  -\psi \right \Vert \geq\varepsilon\right \},
\]%
\[
G_{1}=\left \{  j\in%
\mathbb{N}
:\max_{x\in X}\left \vert T_{j}\left(  \sigma_{x};x\right)  \right \vert
\geq \dfrac{\varepsilon}{2C}\right \}  ,
\]
and
\[
G_{2}=\left \{  j\in%
\mathbb{N}
:\left \Vert T_{j}\left(  Z\right)  -Z\right \Vert \geq \dfrac{\varepsilon
}{2C}\right \}  .
\]
Then we see that $G\subset G_{1}\cup G_{2}.$ So we can write%
\begin{align*}
0 &  \leq \delta_{P}\left(  \left \{  j\in%
\mathbb{N}
:\left \Vert T_{j}\left(  \psi \right)  -\psi \right \Vert \geq \varepsilon
\right \}  \right)  \\
&  \leq \delta_{P}\left \{  j\in%
\mathbb{N}
:\max_{x\in X}\left \vert T_{j}\left(  \sigma_{x};x\right)  \right \vert
\geq \dfrac{\varepsilon}{2C}\right \}  +\delta_{P}\left \{  j\in%
\mathbb{N}
:\left \Vert T_{j}\left(  Z\right)  -Z\right \Vert \geq \dfrac{\varepsilon
}{2C}\right \}  .
\end{align*}
From Lemmas \ref{lemma3} \ref{lemma4}, we get%
\[
st_{P}-\lim_{j\rightarrow \infty}\left \Vert T_{j}\left(  \psi \right)
-\psi \right \Vert =0.
\]
Hence, the proof is concluded.
\end{proof}
The following remark highlights a specific case of this theorem from the literature.


\begin{remark}
Let $X = [a, b]$. Consider the functions $\psi_{1}(y) = 1$, $\psi_{2}(y) = y$, $\psi_{3}(y) = y^{2}$, and $g_{1}(x) = x^{2}$, $g_{2}(x) = -2x$, $g_{3}(x) = 1$. In this case, $\breve{Q}_{x}(y)$ satisfies conditions (\ref{5a}) and (\ref{5b}) as a special case. Consequently, Theorem \ref{theorem2} simplifies to the classical Korovkin theorem presented in \cite{unver-orhan2019}.
\end{remark}

We now present an example demonstrating the existence of a sequence of positive linear operators that satisfies the conditions of Theorem \ref{theorem2} but does not meet the requirements of Theorem 1 in \cite{mhaskar-pai} (see, page 20 of \cite{mhaskar-pai}).

\begin{example}
Let $P = (s_j)$ be the power series method defined as
\[
s_{j} =
\begin{cases}
0 & , \text{if } j = 2k, \\
1 & , \text{if } j = 2k + 1,
\end{cases}
\]
and consider the sequence $(\eta_{j})$ defined by
\[
\eta_{j} =
\begin{cases}
j & , \text{if } j = 2k, \\
0 & , \text{if } j = 2k + 1.
\end{cases}
\]
Now, consider the positive linear operators
\[
T_{j}(f;x) = (1+\eta_{j})\acute{L}_{j}(f;x),
\]
where $T_{0}(f) = 0$ for any $f \in C(X)$. It is straightforward to verify that $T_{j}: C(X) \to C(X)$, where $\acute{L}_{j}(f;x)$ is a sequence of positive linear operators satisfying the hypothesis of Theorem 1 in \cite{mhaskar-pai}.
\end{example}

\section{Practical Outcomes of the Main Results}\label{sec4}

In this section, we explore the consequences of Theorem \ref{theorem2}, demonstrating its versatility and potential in approximation theory. The results established in abstract spaces are general and encompass a wide range of useful theorems, many of which have not been explicitly addressed in the literature. By selecting specific function spaces and operators, we can derive new theorems, highlighting the power of the proposed framework, even without the need for detailed proofs in every instance. We begin by establishing the convergence of $r$-th order generalizations of any sequence of positive linear operators. This generalization offers a broader perspective on operator theory, particularly in cases where positivity is not preserved. Following this, we present a Korovkin-type approximation theorem for periodic functions, utilizing power series statistical convergence. These corollaries underscore the strength of the proposed results in addressing both classical and novel problems in approximation theory.

\subsection{An $r$-th order generalization of operators via $P$-statistical convergence}

We now explore the $r$-th order generalization of a sequence of positive linear operators $\left(  G_{j}\right)$. This concept was initially defined and examined in terms of classical convergence by Kirov and Popova \cite{kirov-popova}. The authors introduced the operators $\left(  G_{j}^{\left[  r\right]  }\right)$ as $r$-th order generalizations of the positive linear operators $\left(  G_{j}\right)$, replacing the function near the point $\xi$ with its $r$-th degree Taylor series. Although linear, these operators $\left(  G_{j}^{\left[  r\right]  }\right)$ do not retain the positivity property. Agratini \cite{agratini} further studied the $r$-th order generalization using $A$-statistical convergence and provided applications to Stancu type operators. Additional developments in this area are documented in \cite{aktasvd, olgun-ince-tasdelen, orkcu, ozarslan-duman-srivastava}.

We can write the sequence of the positive linear operators $G_{j}:$ $C\left[
a,b\right]  \rightarrow C\left[  a,b\right]  $ in the form of
\begin{equation}
G_{j}\left(  \psi;x\right)  =\sum\limits_{k=0}^{\infty}\rho_{j,k}\left(
x\right)  \psi\left(  x_{j,k}\right)  
\label{Ln}%
\end{equation}
with%
\[
G_{j}\left(  1;x\right)  =\sum\limits_{k=0}^{\infty}\rho_{j,k}\left(
x\right)  =1.
\]
A generalization $\left(  G_{j}^{\left[  r\right]  }\right)  $ of the order
$r$-th for $  r\in%
\mathbb{N}
_{0}  $ of $\left(  G_{j}\right)  $ is
defined by%
\begin{equation}
G_{j}^{\left[  r\right]  }\left(  \psi;x\right)  =\sum\limits_{k=0}^{\infty
}\sum\limits_{\nu=0}^{r}\rho_{j,k}\left(  x\right)  \psi^{\left(  \nu\right)
}\left(  \xi\right)  \dfrac{\left(  x-\xi\right)  ^{\nu}}{\nu!}  \label{r.genelleme}%
\end{equation}
for $x\in\left[  a,b\right]$.
It is easy to see that the relation
\[
G_{j}^{\left[  r\right]  }\left(  \psi;x\right)  =G_{j}^{\left[  r\right]
}\left(  \psi\left(  \xi\right)  ;x\right)  =G_{j}\left(  \Gamma_{r}\left(
\psi\left(  \xi\right)  ;x\right)  \right)   
\]
is held for any $x\in\left[  a,b\right]$, where
\[
\Gamma_{r}\left(  \psi\left(  \xi\right)  ;x\right)  =\sum\limits_{\nu=0}%
^{r}\psi^{\left(  \nu\right)  }\left(  \xi\right)  \dfrac{\left(
x-\xi\right)  ^{\nu}}{\nu!},
\]
which is the Taylor's polynomial of degree $r$ of the function $\psi\in C^{r}\left[
a,b\right]  $ in a neighborhood of the point $\xi\in\left[  a,b\right]  $,
where $C^{r}\left[
a,b\right]$ consists of all real-valued functions defined on the interval $[a, b]$ that are continuously differentiable up to order $r$.
Since $\Gamma_{0}\left(  \psi\left(  \xi\right)  ;x\right)  =\psi\left(
x\right)  ,$ we can write%
\[
G_{j}^{[0]}\left(  \psi\left(  \xi\right)  ;x\right)  =G_{j}\left(  \Gamma
_{0}\left(  \psi\left(  \xi\right)  ;x\right)  \right)  =G_{j}\left(
\psi\left(  \xi\right)  ;x\right)  .
\]
Note that, the operator $G_{j}^{\left[  r\right]  }:C^{r}\left[  a,b\right]
\rightarrow C\left[  a,b\right]  $ is linear, but is not positive for $r\ge 1$. If we take $r=0$, the operator $\left(
\ref{r.genelleme}\right)  $ reduces to the operator (\ref{Ln}). If we replace $C\left(  X\right)  $ and ($T_{j})$ with $C^{r}\left[
a,b\right]  $ and ($G_{j}),$ respectively, (\ref{5c}) reduces
to
\begin{equation}
st_{P}-\lim_{j\rightarrow\infty}\left\Vert G_{j}\left(  \psi_{\lambda}\right)
-\psi_{\lambda}\right\Vert =0 \label{2}%
\end{equation}
for $\lambda=1,2,3$, where $\psi_{1}(x)=1$, $\psi_{2}(x)=x$, and $\psi_{3}(x)=x^{2}$.

Now, as a consequence of Theorem \ref{theorem2}, we can obtain the
convergence of $r$-th order generalization of the general positive linear
operators ($G_{j})$ via $P$-statistical convergence in the following theorem.

\begin{corollary}
\label{r.generalization}Let $P$ be a regular power series method and $r\ge 1$ be a fixed integer. Suppose the operators $G_{j}$ and $G_{j}^{[r]}$ are defined as in (\ref{Ln}) and (\ref{r.genelleme}), respectively. If the condition (\ref{2}) is satisfied,
then for any $\psi$ $\in C^{r}\left[  a,b\right]  $ with the property
$\psi^{\left(  r\right)  }\in Lip_{M}\left(  \alpha\right)  $, we have%
\[
st_{P}-\lim_{j\rightarrow\infty}\left\Vert G_{j}^{\left[  r\right]  }\left(
\psi;x\right)  -\psi\right\Vert =0.
\]

\end{corollary}

\subsection{Approximation for periodic functions via $P$-statistical convergence}


The spaces $C_{2\pi}(\mathbb{R})$ and $B(\mathbb{R})$ consist of all $2\pi$-periodic continuous functions and all bounded functions on $\mathbb{R}$, respectively. Korovkin-type approximation theorems for periodic functions utilizing statistical convergence and Abel convergence are discussed in \cite{duman-periodic} and \cite{unver}.

Consider the functions $\psi_1(y) = 1$, $\psi_2(y) = \cos y$, and $\psi_3(y) = \sin y$, with the corresponding coefficients $g_1(x) = 1$, $g_2(x) = -\cos x$, and $g_3(x) = -\sin x$. These settings ensure that $\breve{P}_{x}(y)$ meets the conditions specified in (\ref{5a}) and (\ref{5b}) as a particular instance. Consequently, Theorem \ref{theorem2} simplifies to a specific case of the Korovkin-type approximation theorem.

\begin{corollary}
Let $P$ be a regular power series method and $\left(  G_{j}\right)  $ be a
sequence of positive linear operators from $C_{2\pi}(\mathbb{R})$ into
$B(\mathbb{R})$. Then for any function $\psi\in C_{2\pi}(\mathbb{R})$%
\[
st_{P}-\lim_{j\rightarrow\infty}\left\Vert G_{j}\left(  \psi\right)
-\psi\right\Vert =0
\]
if and only if for all $\lambda=1,2,3$%
\[
st_{P}-\lim_{j\rightarrow\infty}\left\Vert G_{j}\left(  \psi_{\lambda}\right)
-\psi_{\lambda}\right\Vert =0,
\]
where $\psi_{1}(y)=1,$ $\psi_{2}(y)=\cos y,$ $\psi_{3}(y)=\sin y$. \bigskip
\end{corollary}

\section{Conclusion}\label{sec5}
In this paper, we introduced a key result for understanding power series statistical convergence. We established an abstract Korovkin-type approximation theorem in general spaces, demonstrating its utility in handling linear operators, including cases that lack positivity. The conclusions presented, such as the $r$-th order generalization of positive linear operators and the Korovkin-type approximation theorem for periodic functions, illustrate the strength and versatility of our results. The findings presented here contribute to both the theoretical development and practical applicability of approximation theory. By extending classical results and providing new perspectives on convergence via power series methods, this work opens avenues for further research. Potential future directions include exploring the interplay between power series statistical convergence and other summability methods, developing analogous results in modular spaces, and applying these techniques to multivariate functions and operators. Additionally, the framework established in this paper could be extended to study non-linear operators or hybrid methods combining statistical and power series convergence. Another promising direction is the investigation of real-world applications, such as signal processing and numerical analysis, where approximation methods play a critical role.

\section*{Declarations}
\noindent \textbf{Funding:} This research was supported by Selçuk University Coordination Office for Scientific Research Projects (BAP), Project Number 24401223.\\
\textbf{Competing Interests:} The authors have no relevant financial or nonfinancial interests to disclose.\\
\textbf{Data Availability:} No data was used in this study.


\begin{thebibliography}{99}                                                                                               %


\bibitem {agratini} Agratini, O.: Statistical convergence of non-positive
approximation process, Chaos Solitions Fractals 44(11) 977-981 (2011).

\bibitem {aktasvd} Akta\c{s}, R.  S\"{o}ylemez, D. Ta\c{s}delen, F.: Stancu type
generalization of Sz\'{a}sz-Durrmeyer operators involving Brenke-type
polynomials, Filomat, 33  (3), 855-868 
(2019).
\bibitem{alemdar}Alemdar, M. E., Duman, O.: General summability methods in the approximation by Bernstein–Chlodovsky operators. Numerical Functional Analysis and Optimization, 42(5), 497-509 (2021).

\bibitem {altomare}Altomare, F. Campiti, M. Korovkin-type approximaton theory
and its applications, Walter de Gruyter, Berlin-New York, 1994


\bibitem {atlihan}Atl\i han \"{O}. G., Orhan, C.: Matrix summability and
positive linear operators, Positivity, 11 (3), 387-398 (2007).


\bibitem {atlihan-tas} Atlihan, \"{O}. G Ta\c{s}, E.: An Abstract Version of the
Korovkin Theorem via $A$-summation process, Acta Mathematica Hungarica,
145(2), 360-368 (2015).


\bibitem {boos}Boos, J.: Classical and Modern methods in summability. Oxford,
Oxford University Press, 2000.

\bibitem{braha2}Braha, N. L., Mansour, T., Mursaleen, M.: Approximation by Modified Meyer–König and Zeller Operators via Power Series Summability Method. Bulletin of the Malaysian Mathematical Sciences Society, 44(4), 2005-2019 (2021).


\bibitem{top}Çakalli, H., Khan, M. K. Summability in topological spaces. Applied Mathematics Letters, 24(3), 348-352 (2011).

\bibitem{103}Connor, J. S.: The statistical and strong p-Cesaro convergence of sequences. Analysis, 8(1-2), 47-64 (1988).
\bibitem {devore-lorentz}DeVore R.A., Lorentz G.G.: Constructive Approximation, vol. 303,  Springer-Verlag, Berlin, Heidelberg,1993.

\bibitem{demirci4}Demirci, K., Yıldız, S., and Çınar, S.: Approximation of matrix-valued functions via statistical convergence with respect to power series methods. The Journal of Analysis, 1-14.(2022).

\bibitem {duman-periodic} Duman, O.: Statistical approximation for periodic
functions, Demonstratio Mathematica, 4(2003),873-877.

\bibitem {dumanabstract}Duman, O, Orhan, C.: An abstract version of the
Korovkin approximation theorem, Publ. Math. Debrecen, 69, (1-2), 33--46 (2006).

\bibitem {duman2007}Duman, O.: A Korovkin type Approximation Theorems via
$I$-Convergence, Czechoslovak Math. J., 57 (132)(2007), 367-375

\bibitem {fast}Fast,F.: Sur la convergence statistique. Colloq. Math. 2,
241-244, 1951.

\bibitem{102}Fridy, J. A.: . On statistical convergence. Analysis, 5(4), 301-314 (1985).

\bibitem {gadjiev-orhan}Gadjiev, A.D., Orhan, C.: Some approximation theorems
via statistical convergence. Rocky Mountain Journal of Mathematics, 32,
129-138, 2002.

\bibitem{mustafa}Gülfırat, M.: Another example of $P$-ideals,  10th International IFS and Contemporary Mathematics and Engineering Conference, (2024).

\bibitem {kirov-popova} Kirov, G.H.,  Popova,L.: A generalization of the linear
positive operators, Math. Balkanica 7(1993) 149--162.

\bibitem {korovkin}Korovkin, P. P.: On convergence of linear positive operators in the space of continuous functions (Russian). In Doklady Akademii Nauk SSSR (NS), 90, (1953)961.

\bibitem {mhaskar-pai} MhaskarH. N., D. V. Pai, D.: Fundamentals of
ApproximationTheory, CRC Press, Boca Raton, Fla, USA, (2000)

\bibitem {olgun-ince-tasdelen}Olgun, A., \.{I}nce, H.G., Ta\c{s}delen, F.:
Kantorovich-type Generalization of Meyer -K\"{o}n\i g and Zeller Operators via
Generating Functions, An. \c{S}t. Univ. Ovidius Constanta 21(3) (2013) 209-221.

\bibitem {orkcu}\"{O}rk\c{c}\"{u}, M.: Approximation properties of Stancu-type
Meyer -K\"{o}n\i g and Zeller Operators, Hacettepe Journal of Mathematics and
Statistics 42(2)  139-148 (2013).

\bibitem {ozarslan-duman-srivastava} \"{O}zarslan, M.A., Duman, O., Srivastava, H.M.: Statistical approximation results for Kantorovich-type operators
involving some special polynomials, Math. Comput Modelling 48 388-401 (2008).

\bibitem{nilay} Şahin Bayram, N.: Criteria for statistical convergence with respect to power series methods. Positivity, 25(3), 1097-1105 (2021).

\bibitem {sakaoglu-unver} Sakaoglu, \.{I}.,  \"{U}nver, M.: Statistical
approximation for multivariable integrable functions, Miskolc Mathematical
Notes, 13, 485-491 (2012).

\bibitem {salat} Salat, T.: On statistically convergent sequences of real
numbers. Mat. Slovaca., 30 (2), 139-150 (1980).

\bibitem {soy-unver2017}S\"{o}ylemez, D., \"{U}nver, M.: Korovkin type theorems
for Cheney--Sharma Operators via summability methods, Results Math. 73, 1601--1612 (2017)

\bibitem {soylemez-un2021}S\"{o}ylemez, D., \"{U}nver, M.: Rates of Power Series
Statistical Convergence of Positive Linear Operators and Power Series
Statistical Convergence of -Meyer--K\"{o}n\.{i}g and Zeller Operators,
Lobachevskii Journal of Mathematics 42 (2) (2021), 426-434.

\bibitem {Soylemez2021} S\"{o}ylemez, D.: Korovkin type theorem for non-tensor
Bal\'{a}zsz type Bleimann, Butzer and Hahn operators, Mth. Slovaca 72
(1), 153-164 (2022).


\bibitem {srivastava2021} Srivastava, H. M., Ansari, K. J.,  \"{O}zger, F., \"{O}demi\c{s} \"{O}zger, Z.: A link between approximation theory and summability
methods via four-dimensional infinite matrices, Mathematics, 9(16), 1895 (2021).


\bibitem {tas-abstract}Ta\c{s}, E.: "Abstract Korovkin type theorems on
modular spaces by $A$-summability." Mathematica Bohemica 143. 419-430.(2018).

\bibitem {tas-yurdakadim2016}Ta\c{s}, E., Yurdakadim, T.: Approximation to
derivatives of functions by linear operators acting on weighted spaces by
power series method. Computational analysis, Springer Proceedings in
Mathematics and Statistics, 155, 363-372, 2016.


\bibitem {ulucay-un}Ulu\c{c}ay, H., \"{U}nver, M., S\"{o}ylemez, D.: Some
Korovkin type approximation applications of power series methods, Revista de
la Real Academia de Ciencias Exactas, F\'{\i}sicas y Naturales. Serie A.
Matem\'{a}ticas 117(1)(2023)1-24.

\bibitem {unver-khan-orhan}\"{U}nver, M. Khan ,M.K., Orhan, C.: $A$%
-distributional summability in topological spaces, Positivity 18(1)(2014),131-145.


\bibitem {unver}\"{U}nver, M. : Abel transforms of positive
linear operators. In AIP Conference Proceedings 1558(1), 1148-1151(2013).

\bibitem {unver-orhan2019}\"{U}nver, M., Orhan, C.: Statistical convergence
with respect to power series methods and applications to approximation theory.
Journal Numerical Functional Analysis and Optimization. 40 (5), 535-547, 2019.

\bibitem{unver3}Ünver, M., Bayram, N. Ş.: On statistical convergence with respect to power series methods. Positivity, 26(3), 1-13(2022).


\end{thebibliography}
\end{document}